\documentclass[12pt]{amsart}
\usepackage{graphicx}
\usepackage[english]{babel}
\usepackage{amsmath,amsfonts,amssymb}
\usepackage{amsthm}
\usepackage{enumerate}
\usepackage{verbatim}
\usepackage{subfigure}
\usepackage{graphicx,color}
\usepackage{xspace}
\usepackage{fullpage}
\newtheorem{theorem}{Theorem}[section]
\newtheorem{proposition}[theorem]{Proposition}
\newtheorem{lemma}[theorem]{Lemma}

\newtheorem{corollary}[theorem]{Corollary}

\newtheorem*{ack*}{Acknowledgment}
\theoremstyle{remark}
\newtheorem{rem}{Remark}

\newtheorem*{ex*}{Example}
\graphicspath{{Dessins/}}
\def\ie{\emph{i.e.}\xspace}
\def\etal{\emph{et al.}\xspace}
\def\resp{\emph{resp.}\xspace}
\def\Phit{\ensuremath{\Phi_{\!\small\triangle}\!}}
\def\sC{\scalebox{.8}{\ensuremath{C}}\xspace}
\def\tsC{\scalebox{.8}{\ensuremath{\tilde C}}\xspace}
\def\sD{\scalebox{.8}{\ensuremath{D}}\xspace}
\def\sE{\scalebox{.8}{\ensuremath{E}}\xspace}
\def\sP{\scalebox{.8}{\ensuremath{P}}\xspace}
\def\sQ{\scalebox{.8}{\ensuremath{Q}}\xspace}
\def\sR{\scalebox{.8}{\ensuremath{R}}\xspace}
\def\sM{\scalebox{.8}{\ensuremath{M}}\xspace}
\def\sO{\scalebox{.8}{\ensuremath{O}}\xspace}
\def\sT{\scalebox{.8}{\ensuremath{T}}\xspace}
\def\carre{{\scalebox{.4}{\ensuremath{\blacksquare}}\xspace}}
\def\trait{-\xspace}
\def\cA{\ensuremath{\mathcal{A}^\carre}\xspace}
\def\cB{\ensuremath{\mathcal{A}^\trait}\xspace}
\def\cD{\ensuremath{\mathcal{D}}\xspace}
\def\cE{\ensuremath{\mathcal{E}}\xspace}
\def\cTc{\ensuremath{\mathcal{T}^\carre}\xspace}
\def\cTt{\ensuremath{\mathcal{T}^\trait}\xspace}
\def\cTct{\ensuremath{\mathcal{T}^{\bullet\!\!\trait}}\xspace}
\def\cS{\ensuremath{\mathcal{S}}\xspace}
\def\cU{\ensuremath{\mathcal{U}}\xspace}
\def\cV{\ensuremath{\mathcal{V}}\xspace}
\def\cX{\ensuremath{\mathcal{X}}\xspace}
\def\cQ{\ensuremath{\mathcal{Q}}\xspace}

\def\tQ{\ensuremath{\widetilde{Q}}\xspace}
\def\tR{\ensuremath{\widetilde{R}}\xspace}

\def\cff{\ensuremath{t^\carre}\xspace}
\def\cgg{\ensuremath{t^\trait}\xspace}
\def\chh{\ensuremath{t^{\bullet\!\!\trait}}\xspace}
\def\cuu{\ensuremath{u}\xspace}
\def\cvv{\ensuremath{v}\xspace}
\def\cT{\ensuremath{\mathcal{T}}\xspace}

\def\cM{\ensuremath{\mathcal{M}}\xspace}

\def\cF{\ensuremath{\mathcal{F}}\xspace}

\def\cFki{\ensuremath{\cF_i^{(k)}}\xspace}
\def\cGki{\ensuremath{\cG_i^{(k)}}\xspace}
\def\cHki{\ensuremath{\cH_i^{(k)}}\xspace}
\def\Fki{\ensuremath{F^{(k)}_i}\xspace}
\def\cG{\ensuremath{\mathcal{G}}\xspace}

\def\Gki{\ensuremath{G^{(k)}_i}\xspace}
\def\cH{\ensuremath{\mathcal{H}}\xspace}

\def\Hki{\ensuremath{H^{(k)}_i}\xspace}
\def\cD{\ensuremath{\mathcal{D}}\xspace}
\def\cDk{\ensuremath{\cD^{(k)}}\xspace}
\def\cDki{\ensuremath{\cD_i^{(k)}}\xspace}
\def\Dki{\ensuremath{D^{(k)}_i}\xspace}

\def\geq{\geqslant}
\def\leq{\leqslant}

\begin{document}

\title{On symmetric quadrangulations and triangulations}
\author{Marie Albenque \and \'Eric Fusy \and Dominique Poulalhon}
\address{LIX, \'Ecole Polytechnique, 91128 Palaiseau cedex, France}
\thanks{Supported by the European project ExploreMaps -- ERC StG 208471}
\thanks{Email: \texttt{albenque,fusy,poulalhon@lix.polytechnique.fr}}

\begin{abstract}
  This article presents new enumerative results related to symmetric
  planar maps.  In the first part a new way of enumerating rooted
  simple quadrangulations and rooted simple triangulations is
  presented, based on the description of two different quotient
  operations on symmetric simple quadrangulations and triangulations.
  In the second part, based on results of Bouttier, Di Francesco and
  Guitter and on quotient and substitution operations, the series of
  three families of symmetric quadrangular and triangular dissections
  of polygons are computed, with control on the distance from the
  central vertex to the outer boundary.
\end{abstract}

\maketitle
\section*{Introduction}
Enumeration of families of plane maps, that is, plane embeddings of
graphs, has received a lot of attention since the 60's; several
methods can be applied: the recursive method introduced by
Tutte~\cite{Tu63}, the random matrix method introduced by Br\'ezin et
al~\cite{Bre}, and the bijective method introduced by Cori and
Vauquelin~\cite{CoriVa} and Schaeffer~\cite{S-these}.  In the first
part of this note, we show another method for the enumeration of
rooted simple quadrangulations and triangulations based on quotienting
symmetric simple versions of them.  Historically, the enumeration of
symmetric maps of order $k$ (\ie, such that a rotation of order $k$
fixes the map) was reduced to the enumeration of rooted maps via a
quotient argument, a method used by Liskovets~\cite{Li78}.  We proceed
in the reverse way, namely we use two quotient operations on symmetric
simple quadrangulations and triangulations to build in each case an
algebraico-differential equation (Equations~\eqref{eq:q}
and~\eqref{eq:t}) satisfied by the generating series of rooted
corresponding simple maps, which can be explicitly solved to obtain
the formulas for the number of rooted simple quadrangulations (due to
Tutte~\cite{Tu63} and bijectively proved by Schaeffer~\cite{S-these})
and of rooted simple triangulations (due to Tutte~\cite{Tu62} and
bijectively proved by Poulalhon and Schaeffer~\cite{PoSc06}).  One
quotient operation is classical and is decribed in
Section~\ref{sec:clasquo}; the other quotient operation is new and, as
described in Section~\ref{sec:ori}, relies deeply on the existence and
properties of $\alpha$-orientations; the new equations for generating
series of simple quadrangulations and triangulations are derived and
solved in Section~\ref{sec:equation}.

The results in the second part are expressions of the series of
several families of symmetric quadrangular and triangular dissections
with control on the distance from the central vertex to the outer
boundary.  We recall that symmetric dissections have been counted
according to the number of inner faces by
Brown~\cite{Brown_trig,Brown_quad} using the recursive method
(Liskovet's quotient method~\cite{Li78} can also be applied, reducing
the enumeration to rooted quadrangular dissections). Our approach,
developed in Sections~\ref{sec:dist} and~\ref{sec:dist_trig}, relies
on the quotient method and substitution operations combined with
results by Bouttier \etal~\cite{BoDFGu03,BoGu12}, which express the
series of quadrangulations or triangulations with a marked vertex and
marked edge at prescribed distance from each other.
Our expressions illustrate again the property that the series
expression of a ``well behaved'' map family \cM refined by a distance
parameter $d$ is typically expressed in terms of the $d$th power of an
algebraic series of singularity type $z^{1/4}$ (implying that
asymptotically the distance parameter $d$ on a random map of size $n$
in \cM converges in the scale $n^{1/4}$ as a random variable).

\section{Plane maps, symmetry and classical quotient}\label{sec:clasquo}

\subsection{Triangulations, quadrangulations and dissections}
A \emph{plane map} is a connected graph embedded in the plane up to
continuous deformation; the unique unbounded face of a plane map is
called the \emph{outer face}, the other ones are called \emph{inner
  faces}. Vertices and edges are also called outer if they belong to
the outer face and inner otherwise. A map is said to be \emph{rooted}
if an edge of the outer face is marked and oriented so as to have the
outer face on its left. This edge is the \emph{root edge}, and its
origin is the \emph{root vertex}. A map is \emph{pointed} if one of
its \emph{inner} vertices is marked.  For any map \sM, we denote by
$\cV(\sM)$, $\cF(\sM)$, $\cE(\sM)$ its sets of vertices, faces and
edges, and by $v(\sM)$, $f(\sM)$, $e(\sM)$ their cardinalities.

Triangulations and quadrangulations are respectively maps with all
faces of degree 3 or 4, and to avoid the degenerated cases, the outer
face is required to be a simple cycle. For $k\geq 1$ and $d\geq3$, a
\emph{$d$-angular dissection of a $k$-gon} or \emph{$d$-angular
  $k$-dissection} is a map whose outer face contour is a simple cycle
of length $k$, and with all inner faces of one and the same
degree~$d$. A dissection is said to be \emph{triangular} if $d$
equals~3, \emph{quadrangular} if $d$ equals~4. Observe that
quadrangular $k$-dissections can only exist for even $k$.

A map is said to be \emph{simple} if it has no multiple edges; a
$d$-angular $k$-dissection is called \emph{irreducible} if the
interior of every cycle of length at most $d$ is a face.

\subsection{Symmetric maps and classical quotient}\label{sub:clasquo}
 For $k\geq 2$, a dissection \sD is said to be \emph{$k$-symmetric} if 
its plane embedding (conveniently deformed) is invariant by a
$2\pi/k$-rotation centered at a vertex -- called the \emph{center}
of~\sD.

As observed by Liskovets~\cite{Li78}, any two semi-infinite straight
lines starting from the center and forming an angle of $2\pi/k$
delimit a sector of~\sD. When keeping only this sector and pasting
these two lines together, we obtain a plane map, called the
\emph{$k$-quotient map} of \sD. In other words, the $2\pi/k$-rotation
defines equivalence relations on the sets $\cV(\sD)$ and $\cE(\sD)$,
and the quotient map of \sD is the map in which equivalent vertices
and equivalent edges are
identified. Figure~\ref{fig:classical_quotient} shows the example of
two symmetric dissections of an hexagon and their quotients.  Denote
by $o(D)$ the degree of the outer face of a dissection.  The following
lemma is straightforward:

\begin{figure}
  \centering
  \subfigure[\label{fig:3-sym}]{\includegraphics[page=1,scale=0.7]{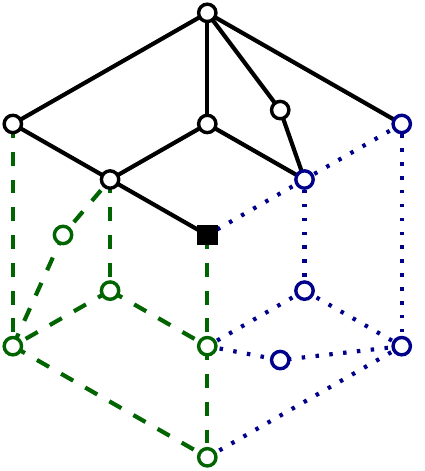}}\quad
  \subfigure[\label{fig:3-quo}]{\includegraphics[page=2,scale=0.7]{hexagone_sym}}\qquad\vline\qquad
  \subfigure[\label{fig:2-sym}]{\includegraphics[page=3,scale=0.7]{hexagone_sym}}\qquad
  \subfigure[\label{fig:2-quo}]{\includegraphics[page=4,scale=0.7]{hexagone_sym}}
  \caption{Examples of symmetric 6-dissections. A 3-symmetric
    quadrangular dissection~\subref{fig:3-sym}. Its
    3-quotient~\subref{fig:3-quo} is a pointed quadrangular 2-dissection. 
    A 2-symmetric triangular dissection~\subref{fig:2-sym}. Its
    2-quotient~\subref{fig:2-quo} is a pointed triangulation.}
  \label{fig:classical_quotient}
\end{figure}

\begin{lemma}\label{lem:quotient_count}
  For $k\geq 2$, let $\sD$ be a $k$-symmetric dissection,
  and $\sE$ its $k$-quotient; we have:
  \[
  v(\sD) - 1 = k\, (v(\sE)-1) , \quad e(\sD) = k\,e(\sE), \quad f(\sD)
  - 1 = k\, (f(\sE)-1), \quad o(\sD)=k\,o(\sE).
  \]
\end{lemma}

The quotient operation clearly preserves the degrees of vertices and
faces, hence quotients of $d$-angular dissections are $d$-angular
dissections. More precisely, as any $k$-symmetric dissection is
implicitly pointed (at the center), its $k$-quotient is a pointed
dissection.  We define the \emph{radial distance} $r(\sD)$ of a
pointed dissection \sD as the distance between its marked vertex and
the outer face boundary (for instance, the example of
Figure~\ref{fig:3-quo} has radial distance~2 while that of
Figure~\ref{fig:2-quo} has radial distance~1). We also $\ell(\sD)$ the
length of a shortest cycle strictly enclosing $u$.
The following lemma 
summarizes the relations between distances in $k$-symmetric
dissections and their $k$-quotients:

\begin{lemma}\label{lem:quotient_distance}
  For $k\geq 2$, let $\sD$ be a $k$-symmetric dissection,
  and $\sE$ its $k$-quotient; we have:
  \[
  r(\sD)=r(\sE) \quad\text{and}\quad \ell(\sD)=k\,\ell(\sE).
  \]
\end{lemma}
\begin{proof}
  Let \sP be a pointed dissection, with pointed vertex $u$.  A
  labelling-function is a function $\lambda: \cV(\sP) \to\mathbb{Z}$
  such that $\lambda(u)=0$ and for all adjacent vertices $v$ and $v'$,
  $|\lambda(v)-\lambda(v')|\leq 1$. It is easy to check that the
  function $\delta_{P}$ giving the distance from $u$, called
  distance-labelling, is the unique labelling-function such that each
  vertex $v\neq u$ has a neighbour of smaller label.  With this
  characterization, it is straightforward that $\delta_{E}$ is the
  quotient of $\delta_{D}$; in particular, the radial distance is
  the same in \sD as in~\sE.

  The second statement is equivalent to the following one: one among
  the cycles of minimal length stricly enclosing the center is itself
  symmetric. Indeed, let \sC be a cycle of minimal length $\ell(\sD)$
  strictly enclosing the center that is minimal for inclusion, and
  suppose that \sC is asymmetric. Let \tsC be its image by the
  $\pi/k$-rotation; as they both enclose the center, \sC and \tsC are
  intersecting cycles and hence define two other cycles enclosing the
  center with total length $2 \ell(\sD)$, hence each of length
  $\ell(\sD)$; one of these two cycles is included in \sC and \tsC,
  contradiction.
\end{proof}

\subsection{Symmetric simple quadrangulations and triangulations}

A pointed dissection is called \emph{quasi-simple} if the pointed
vertex lies strictly in the interior of every 1-cycle or 2-cycle; in
particular each vertex can carry at most one loop, otherwise the two
loops would form a 2-cycle that does not contain the pointed vertex.
Unlike that of simple dissections, the family of quasi-simple
dissections is stable under quotient:
\begin{lemma}\label{lem:quotient_quasisimple}
  For $k\geq 2$, let $\sD$ be a $k$-symmetric dissection and $\sE$ its
  $k$-quotient -- both canonically pointed.  Then $\sD$ is
  quasi-simple if and only if $\sE$ is quasi-simple.
\end{lemma}
\begin{proof}
  This just follows from the fact that any 2-cycle of $\sD$ not
  enclosing the pointed vertex yields a 2-cycle of $\sE$ not enclosing
  the pointed vertex, and reciprocally.
\end{proof}
Note also that, according to Lemma~\ref{lem:quotient_distance}, no
cycle enclosing the center of a $k$-symmetric dissection \sD can be of
length smaller than~$k$. In particular, \sD has no loop, and no
2-cycle enclosing the center if $k>2$. Hence if $k>2$, $k$-symmetric
quasi-simple dissections are indeed simple.

\paragraph{Symmetric simple quadrangulations}
Lemma~\ref{lem:quotient_count} implies that quadrangulations, having
outer degree~4, may only be 2- or 4-symmetric. Moreover, their
quotients are quadrangular dissections, hence they are bipartite,
which prevents them of containing loops, in particular the outer face
has length at least~2. Hence quadrangulations may only be 2-symmetric
(which we simply call \emph{symmetric}), and quotients of symmetric
quadrangulations are pointed quadrangular 2-dissections.  This implies
also that quasi-simple symmetric quadrangulations are indeed
simple. Hence the following proposition is a direct consequence of
Lemma~\ref{lem:quotient_quasisimple}:

\begin{proposition}\label{prop:clasquoquad}
  The 2-quotient is a one-to-one correspondence between symmetric
  simple quadrangulations with $2n$ inner faces and quasi-simple
  pointed 2-dissections with $n$ inner faces.
\end{proposition}

\paragraph{Symmetric simple triangulations}
Similarly, triangulations may only be 3-symmetric (which we simply
call \emph{symmetric}), and the 3-quotient of a symmetric
triangulation is a pointed triangular 1-dissection.  Note that a
quasi-simple symmetric triangulation is simple (indeed, by
Lemma~\ref{lem:quotient_distance}, a 3-symmetric dissection $\sD$
satisfies $\ell(\sD)\geq 3$).  Hence:

\begin{proposition}\label{prop:quoclastrig}
  The 3-quotient is a one-to-one correspondence between symmetric
  simple triangulations with $3n$ inner faces and quasi-simple pointed
  triangular 1-dissections with $n$ triangular faces, where $n$ is any
  positive odd integer.
\end{proposition}

\section{Another quotient for symmetric simple quadrangulations and
  triangulations}\label{sec:ori}

In this section, we only consider simple quadrangulations and
triangulations, and show that a specific quotient can be defined for
these families, using their characterization in terms of
$d$-orientations.

\subsection{The minimal 2- or 3-orientation}\label{sub:ori}

An \emph{orientation} of a plane map is the choice of an orientation
for each of its \emph{inner} edges.  An orientation without
counterclockwise cycle is said to be \emph{minimal}. A $d$-orientation of a
map is an orientation such that the outdegree of each inner vertex is
equal to $d$, while outer vertices have outdegree equal to zero.

An important property of simple quadrangulations and triangulations is
that they can be characterized in terms of $d$-orientations:

\begin{proposition}[\cite{Schnyder},\cite{DeOs01},\cite{Felsner2004}]
  A quadrangulation is simple if and only if it admits a
  2-orientation.  A triangulation is simple if and only if it admits a
  3-orientation.  In both cases, there exists a unique minimal such
  orientation.
\end{proposition}

Figure~\ref{fig:symquad} and~\ref{fig:symtrig} show two examples of a
simple quadrangulation and a simple triangulation respectively endowed
with their minimal 2- or 3-orientation.

Now let \sM be a $k$-symmetric map endowed with a $d$-orientation \sO;
then the image of \sO by the $2\pi/k$-rotation is clearly a
$d$-orientation, and it is minimal if \sO is minimal. Hence:

\begin{lemma}
  The unique minimal 2-orientation (\resp 3-orientation) of a
  symmetric simple quadrangulation (\resp triangulation) is itself
  symmetric.
\end{lemma}

\subsection{Leftmost paths}

Let \sM be a plane map endowed with an orientation.  For each inner
vertex $u$ of \sM, a \emph{leftmost path starting at $u$} is a maximal
oriented path \sP starting at $u$ such that for any triple $v,v',v''$
of successive vertices along \sP, $(v',v'')$ is the first outgoing
edge after $(v,v')$ in clockwise order around $v'$.

\begin{lemma}\label{lem:leftmostpath}
  If \sM is a simple quadrangulation or triangulation endowed with its
  minimal 2- or 3- orientation, \sP is necessarily a simple path
  ending at one outer vertex of~\sM.
\end{lemma}

\begin{proof}
  Assume by contradiction that \sP is self intersecting and let
  $v_0,v_1,\ldots, v_l=v$ be a sequence of vertices of \sP explored
  successively in this order and forming a cycle. Since the
  orientation \sO of \sM is chosen to be minimal, this cycle is
  necessarily clockwise. Hence, since \sP is a left-most path, the
  outgoing edges of each vertex of the cycle lie in the interior
  region of the cycle (except eventually the other outgoing edges of
  $v_0$, if $v_0$ is the starting point of the path), hence denoting
  by \sR the enclosed region, $e(\sR)$ is at least $2v(\sR)-1$ (\resp
  $3v(\sR)-2$). Euler relation leads to a contradiction: if the length
  of the cycle is $\ell$, Euler formula implies that $2v(\sR) = e(\sR)
  + 2 + \ell/2$ (\resp $3v(\sR) = e(\sR) + 2 + \ell$). Note that this
  is in accordance with the well-known property that inner edges of a
  simple triangulation (\resp quadrangulation) can be partitionned
  into three (\resp two) spanning trees (see~\cite{Schnyder}).
\end{proof}

In the next two subsections, we use leftmost paths starting at the
center to decompose symmetric simple quadrangulations and
triangulations in 2 (\resp 3) sectors, and show how this particular
sector can be sticked together again in a non-classical way to obtain
a simple quadrangulation or triangulation.

\subsection{A new way of quotienting symmetric simple quadrangulations}

\begin{figure}[t]
  \centering
  \def\etalonminipage{.25\linewidth}
  \def\etalon{10em}
  \begin{minipage}[c]{.3\linewidth} \centering
    \subfigure[\scriptsize Symmetric simple quadrangulation]{\qquad
      \includegraphics[height=\etalon, page=1]{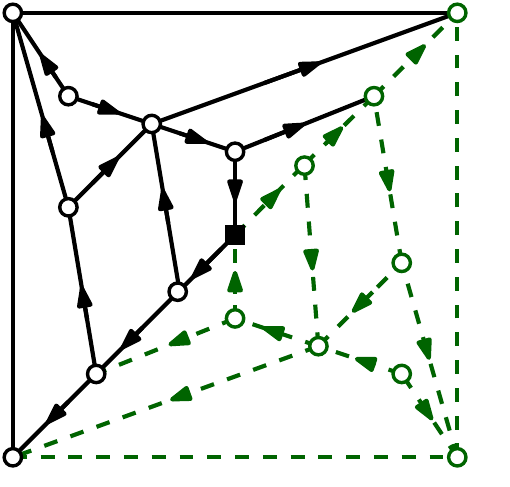}\quad
      \label{fig:symquad}
    }
  \end{minipage}
  \begin{minipage}[c]{3em}\centering
    $\nearrow$
    \vskip5em
    $\searrow$
  \end{minipage}
  \begin{minipage}[c]{\etalonminipage}\centering
    \subfigure[\scriptsize Classical identification]{
      \includegraphics[height=\etalon, page=3]{Quotient}
      \label{fig:ident-classique}
    }\\
    \subfigure[\scriptsize New identification]{
      \includegraphics[height=\etalon, page=2]{Quotient}
      \label{fig:idnew}
    }
  \end{minipage}
  \begin{minipage}[c]{3em}\centering
    $\longrightarrow$
    \vskip10em
    $\longrightarrow$
  \end{minipage}
  \begin{minipage}[c]{\etalonminipage}\centering
    \subfigure[\scriptsize Classical quotient]{
      \includegraphics[height=\etalon, page=5]{Quotient}
      \label{fig:quoclassique}
    }\\
    \subfigure[\scriptsize New quotient]{
      \includegraphics[height=\etalon,page=4]{Quotient}
      \label{fig:quonew}
    }
  \end{minipage}
  \label{fig:subfigureExample}
  \caption{Classical 2-quotient
    \subref{fig:ident-classique}, \subref{fig:quoclassique} and the
    new quotient \subref{fig:idnew}, \subref{fig:quonew} of a symmetric simple quadrangulation endowed
    with its minimal 2-orientation \subref{fig:symquad}.}
\end{figure}

Let \sQ be a symmetric simple quadrangulation, $u$ the central vertex,
$e_1$, $e_2$ its two outgoing edges, and $P_1=(u=v_0,v_1,\ldots,v_p)$,
$P_2=(u=w_0,w_1,\ldots,w_p)$ the leftmost paths of $e_1$ and $e_2$
respectively. Clearly $P_1$ and $P_2$ map to one another by the
$\pi$-rotation.  Hence $P_1$ and $P_2$ cannot meet except at their
starting point $u$, otherwise they would meet twice, which would imply
the existence of an oriented cycle without outgoing edges, leading to
the same contradiction as in Lemma~\ref{lem:leftmostpath}.

Let us cut \sQ along $P_1\cup P_2$ to split \sQ into two isomorphic
dissections, see Figure~\ref{fig:idnew}, and define $\sQ_1$ as the one
with clockwise contour $u, v_1, v_2,\dots, w_1$. If $\sQ_1$ is a
quadrangulation, we set $\Phi(\sQ):=\sQ_1$ and mark the edge
$(u,v_1)$. Otherwise, for any $i\leq p-2$, we identify in $\sQ_1$
vertices $v_{i+2}$ with $w_i$, and merge corresponding edges; this
defines the map $\Phi(\sQ)$, in which we then mark the edge $(v_1,v_2)$.

Concerning orientations, the identification of $v_{i+2}$ with $w_i$
creates an orientation conflict only when merging $(u,v_1)$ with
$(v_1,v_2)$. We choose to orient the merged edge from $v_1$ to $v_2$.
With this convention, $\Phi(\sQ)$ is naturally endowed with its minimal
$2$-orientation and the leftmost path of the marked edge $(v_1,v_2)$
is $(v_1,v_2,\ldots,v_p)$ (to justify that the $2$-orientation of $\Phi(\sQ)$ thus obtained
is minimal, one just has to observe that no oriented path goes from some vertex $w_i$
to some vertex $v_j$, hence when doing the identifications of vertices 
no counterclockwise circuit is created).   
It is then easy to describe the inverse
mapping 
and to obtain:

\begin{theorem}\label{th:bij}
  The mapping $\Phi$ is a one-to-one correspondence between symmetric
  simple quadrangulations with $2n$ inner faces and simple
  quadrangulations with $n$ inner faces and a marked edge.
\end{theorem}

\subsection{A new way of quotienting symmetric simple triangulations}

\begin{figure}[t]
\centering
\def\etalon{12em}
\begin{minipage}[c]{.32\linewidth}
  \subfigure[\scriptsize Symmetric simple triangulation]{\quad
    \includegraphics[height=\etalon,page=1]{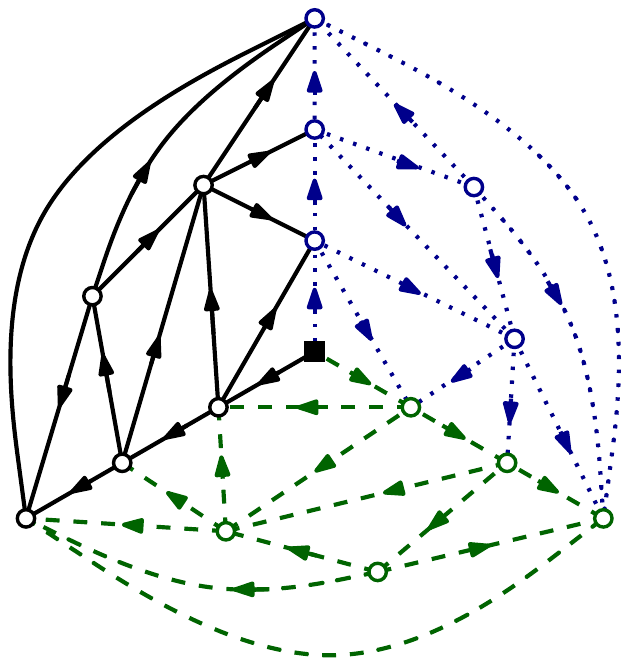}\quad
    \label{fig:symtrig}
  }
\end{minipage}
\begin{minipage}[c]{2em}\centering
  $\nearrow$  \vskip5em  $\searrow$
\end{minipage}
\begin{minipage}[c]{.25\linewidth}
  \subfigure[\scriptsize Classical identification]{\quad
    \includegraphics[height=\etalon,page=3]{QuotientTrig.pdf}
    \label{fig:ident-classique_trig}\hspace{-2em}
  }\\
  \subfigure[\scriptsize New identification]{\quad
    \includegraphics[height=\etalon,page=2]{QuotientTrig.pdf}\
    \label{fig:idnew_trig}\hspace{-2em}
  }
\end{minipage}
\begin{minipage}[c]{2em}\centering
  $\longrightarrow$  \vskip10em  $\longrightarrow$
\end{minipage}\hspace{-1em}
\begin{minipage}[c]{.25\linewidth}
  \subfigure[\scriptsize Classical 3-quotient]{
    \includegraphics[height=\etalon,page=5]{QuotientTrig.pdf}
    \label{fig:quoclassique_trig}
  }\\
  \subfigure[\scriptsize New quotient]{\qquad
    \includegraphics[height=\etalon,page=4]{QuotientTrig.pdf}
    \label{fig:quonew_trig}
  }
\end{minipage}

\label{fig:subfigureExample_trig}
\caption{Classical 3-quotient
  \subref{fig:ident-classique_trig}, \subref{fig:quoclassique_trig} and the
  new quotient \subref{fig:idnew_trig}, \subref{fig:quonew_trig} of a symmetric
simple triangulation endowed
  with its minimal 3-orientation \subref{fig:symtrig}.}
\end{figure}

Let \sT be a simple symmetric triangulation endowed with its minimal
3-orientation. Let $u$ be its central vertex, $e_1$, $e_2$ and $e_3$
its three outgoing edges, and $\sP_1=(u=v_0,v_1,\ldots,v_p)$,
$\sP_2=(u=w_0,w_1,\ldots,w_p)$ and $\sP_3(u=x_0,x_1,\ldots,x_p)$ the
leftmost paths of $e_1$, $e_2$ and $e_3$ respectively. Clearly $\sP_1$,
$\sP_2$ and $\sP_3$ maps respectively onto $\sP_2$, $\sP_3$ and $\sP_1$ by the
$2\pi/3$-rotation centered at $u$.  Hence they cannot meet 
except at their starting point $u$.  Let us cut \sT along $\sP_1\cup
\sP_2\cup \sP_3$ to split \sT into three isomorphic dissections, see
Figure~\ref{fig:idnew_trig}, and define $\sT_1$ as the one with clockwise
contour $u, v_1, v_2,\dots, w_1$. If $\sT_1$ is a triangulation, we set
$\Phit(\sT):=\sT_1$ and mark the edge $(u,v_1)$. Otherwise, for any $i\leq
p-2$, we identify in $\sT_1$ vertices $v_{i+2}$ with $w_i$, and merge
corresponding edges; this defines the map $\Phit(\sT)$, in which we then
mark the edge $(v_1,v_2)$.

Concerning orientations, the identification of $v_{i+2}$ with $w_i$
creates an orientation conflict only when merging $(u,v_1)$ with
$(v_1,v_2)$. We choose to orient the merged edge from $v_1$ to $v_2$.
With this convention, $\Phi(Q)$ is naturally endowed with its minimal
$3$-orientation and the leftmost path of the marked edge $(v_1,v_2)$
is $(v_1,v_2,\ldots,v_p)$. It is then easy to describe the inverse
mapping and to obtain:

\begin{theorem}\label{thm:bijtrig}
  The mapping $\Phit $ is a one-to-one correspondence between
  symmetric simple triangulations with $3n$ inner faces and simple
  triangulations with $n$ inner faces and one marked edge, for any
  positive odd integer~$n$.
\end{theorem}

\section{Simple quadrangulations and triangulations via symmetric
  ones}\label{sec:equation}

\subsection{Getting a functional equation for simple quadrangulations}

Proposition~\ref{prop:clasquoquad} and Theorem~\ref{th:bij} describe
two bijections between symmetric simple quadrangulations and two other
families that are hence also in one-to-one correspondence:
quasi-simple pointed 2-dissections and simple quadrangulations with a
marked edge. Because 2-dissections are bipartite, each pointed
2-dissection corresponds to two different rooted pointed
2-dissections. Similarly, any edge of a simple quadrangulation has an
implicit orientation given by its minimal 2-orientation, hence a
simple quadrangulation with a marked edge corresponds to two distincts
quadrangulations with a marked oriented edge, that can be seen as
rooted simple quadrangulations with a marked face (possibly the outer
one). Hence we obtain:

\begin{corollary}\label{cor}
  Rooted simple quadrangulations with $n$ inner faces and a marked
  face are in one-to-one correspondence with rooted quasi-simple
  pointed 2-dissections with $n$ inner faces.
\end{corollary}

This correspondence allows us to get a functional equation for the
generating series of the family \cQ of non-degenerated (meaning, with at least $2$ faces) 
rooted simple
quadrangulations. Let $q(x)=\sum_{n\geq 2}q_nx^n$ be the series of \cQ
according to the number of faces (including the outer one). The
generating series of rooted simple quadrangulations with a marked face
is then equal to $q'(x)$. We want to express also the family $\cD_4$ of
rooted quasi-simple 2-dissections in terms of~\cQ.

\begin{figure}
  \centering
  \subfigure[Nested
  2-cycles]{\includegraphics[scale=.9, page=3]{decomp_quad}}
  \qquad\quad 
  \subfigure[Element of
  $\mathcal{A}^{\scriptscriptstyle{\blacksquare}}$]{\includegraphics[scale=.9,
    page=1]{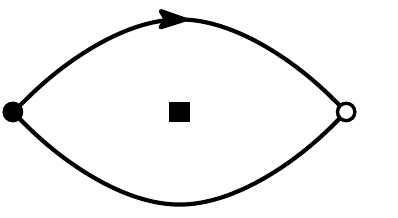}}
  \qquad\quad
  \subfigure[Element of
  $\mathcal{A}^{\scriptscriptstyle{-}}$]{\includegraphics[scale=.9, page=2]{decomp_quad}}
  \caption{The decomposition of a quasi-simple quadrangular 2-dissection.}
  \label{fig:decas}
\end{figure}

Observe that for any $\sD\in \cD_4$, its separating 2-cycles are
nested and therefore ordered from innermost to outermost. This yields
a decomposition of $\sD_4$ as a sequence of components. Let \cA (\resp
\cB) be the family of rooted 2-dissections with a marked inner vertex
(\resp with a marked inner edge) and where the unique 2-cycle is the
outer face contour.  We can cut \sD along the nested 2-cycles to
obtain a pointed 2-dissection in \cA and a sequence of maps in \cB
(see Figure~\ref{fig:decas}).  Denoting respectively by $d(x)$,
$a^\carre(x)$, $a^\trait(x)$ the generating series of \cD, \cA, and
\cB according to the number of quadrangular faces, this gives:
\begin{equation}\label{eq:d}
d(x)=\frac{a^\carre(x)}{1-a^\trait(x)}.
\end{equation}

Deleting the non-root outer edge of a rooted 2-dissection gives a
rooted quadrangulation (possibly degenerated). Taking into account the
marked inner vertex or edge, respectively chosen among the $n$ inner
vertices and $2n-1$ inner edges, we get:
\begin{equation}\label{eq:a}
\begin{cases}
a^\carre(x) ~=~ 2x + \sum_{n\geq 2} n q_n x^n ~=~ 2x + xq'(x) &\\
a^\trait(x) ~=~ 2x + \sum_{n\geq 2} (2n-1) q_n x^n ~=~ 2x + 2xq'(x) - q(x)&
\end{cases}
\end{equation}

Putting together Equations~\eqref{eq:d} and~\eqref{eq:a},
Corollary~\ref{cor} implies:
\begin{proposition}\label{prop:compquad}
  The series $q(x)$ satisfies the following equation:
\begin{equation}\label{eq:q}
  x \cdot [2q'(x)^2+3q'(x)+2] =  q'(x) \cdot [1+q(x)].
\end{equation}
\end{proposition}

From this equation, written as $q'=x(2+2q'^2+3q')/(1+q)$, 
one readily extracts the development of $q(x)$
incrementally:
\[
q(x)=x^2+2x^3+6x^4+22x^5+91x^6+408x^7+1938x^8+\dots
\]
As shown next, the exact expression of the coefficients (first obtained
by Tutte from the recursive method~\cite{Tu63} and subsequently
recovered by Schaeffer~\cite{S-these} using a bijection with ternary
trees) can also be recovered from~\eqref{eq:q}:

\begin{corollary}
  For $n\geq 1$, the number of rooted simple quadrangulations with $n$
  faces is equal to:
  \[
\frac{4(3n)!}{n!(2n+2)!}.
  \]
  Equivalently, the series $q(x)$ is expressed as $q(x) = x
  [\alpha(x)-2] [1-\alpha(x)]$, where $\alpha\equiv\alpha(x)$ is the
  series of rooted ternary trees, specified by $\alpha = 1+x\alpha^3$.
\end{corollary}
\begin{proof}
  Equation~\eqref{eq:q} above admits a unique power series solution
  that is equal to $0$ at $0$, hence it suffices to check that $f
  \equiv f(x) := x [\alpha(x)-2] [1-\alpha(x)]$ is indeed solution
  of~\eqref{eq:q}; note that $x$ and $f(x)$ have rational expressions
  in terms of $\alpha$, hence this also holds for $f'(x)$, since
  $f'(x)=\frac{\mathrm{d}f}{\mathrm{d}\alpha} /
  \frac{\mathrm{d}x}{\mathrm{d}\alpha}$.  We have:
  \[
  x=\frac{\alpha-1}{\alpha^3},\qquad
  f(x)=\frac{(\alpha-1)^2 (2-\alpha)}{\alpha^3},\quad
  \text{and}\quad f'(x)=2\alpha-2.
  \]
  Plugging these expressions in Equation~\eqref{eq:q}, the two
  handsides coincide, which concludes the proof.

  As Alin Bostan showed us, Equation~\eqref{eq:q} can however be
  solved directly, without guessing the solution. Let $r \equiv r(x) =
  q'(x)$, Equation~\eqref{eq:q} rewrites:
  \[
  1 + q = x \cdot \frac{2r^2+3r+2}{r},
  \]
  hence, taking the derivative:
  \begin{equation}\label{eq:deriv}
  r = \frac{2r^2+3r+2}{r}  + 2 x r' \frac{r^2-1}{r^2}, \qquad
  \text{\ie} \qquad  
  (r+1) \cdot \left[ r(r+2) + 2 x r'(r-1) \right] =  0,
  \end{equation}
  or, assuming $r\neq 2$:
  \begin{equation}\label{eq:alin}
  (r+1) \cdot \frac{r^2}{(r+2)^2} \cdot \frac{\mathrm{d}}{\mathrm{d}x}
  \left(\frac{x(r+2)^3}{r}\right) = 0 .
  \end{equation}
  Assuming now that $r\neq0$ and $r\neq 1$, we get $x (r(x)+2)^3 -
  c r(x) = 0 $ for some constant $c$. 
  On the other hand, Equation~\eqref{eq:deriv} implies that 
  \[
  \frac{\mathrm{d}}{\mathrm{d}x}\left(4q(x) - 2xr(x) + xr(x)^2\right)
  = 0,
  \qquad \text{hence} \qquad q(x) = \frac12 xr(x) + \frac14 xr(x)^2
  + c'
  \]
  for some constant~$c'$; $q(0) = r(0) = 0$ implies that $c'= 0$. Thus
  the non affine solutions of Equation~\eqref{eq:q} have the following form:
  \[
  c + x \cdot (\alpha_c(x)-1)(2-\alpha_c(x)) \quad\text{where}\quad x
  \alpha_c(x)^3 = c(\alpha_c(x)-1).
  \]
  Initial conditions for $q$ imply that $c=1$.
\end{proof}

\subsection{Getting a functional equation for simple triangulations}\label{sub:triangulation}

The case of simple triangulations is very similar to the one described
above for quadrangulations; only the decomposition is
more complicated.

Proposition~\ref{prop:quoclastrig} and Theorem~\ref{thm:bijtrig}
describe bijections between symmetric simple triangulations and two
families that are hence also in one-to-one correspondence:

\begin{corollary}\label{cor:trig}
  Simple triangulations with $2n$ faces and one marked edge are
  in one-to-one correspondence with quasi-simple pointed 1-dissections
  with $2n$ faces.
\end{corollary}

Let $\cT$ be the family of rooted simple triangulations and let
$t(x)=\sum_{n\geq 1}t_nx^n$ be the series of $\cT$ according to half
the number of faces, \ie $t_n$ is the number of rooted simple
triangulations with $2n$ faces. We want to use
Corollary~\ref{cor:trig} to derive a functional equation for $t(x)$,
which requires to first express the family \cS of simple
triangulations with a marked edge and the family $\cD_3$ of quasi-simple
pointed 1-dissections in terms of~\cT.

According to the Schnyder tree decomposition of simple triangulations,
each edge is canonically associated to one of the three outer
edges, hence simple triangulations with a marked edge are exactly
rooted simple triangulations with a marked edge chosen among the $n$
edges in the tree rooted at the root edge. Hence:
\begin{equation}\label{eq:trigleft}
  s(x) = \sum_{n\geq 1} n t_n x^{n} = xt'(x).
\end{equation}

The enumeration of quasi-simple pointed triangulations can be carried
out similarly to quasi-simple pointed quadrangulations, but
extra care has to be taken in order to deal with special conditions
for both loops and 2-cycles. Observe that any 1-dissection \sD in
$\cD_3$ is implicitly rooted, and that removing the outer loop yields
either a unique edge or a quasi-simple pointed triangular 2-dissection that
can be canonically rooted with the same root vertex as \sD, 
with the additional constraint that there is no loop
incident to the root vertex; let us call $\cU$ this family and $\cuu$
its generating series according to half the number of triangular
faces. Then the generating series $d_3$ of $\cD_3$ is given by:
\[
  d_3(x) = x \cdot [1+\cuu(x)]. 
\]

Note that removing the non-root edge incident to the outer face of a
triangular 2-dissection produces a rooted triangulation with one face
less.  Hence families of rooted triangular 2-dissections (such as \cU)
and families of rooted triangulations can be naturally identified.  To
decompose elements of \cU, let \cV be the family of quasi-simple
triangular 2-dissections and \cTc, \cTt and \cTct be respectively the
families of rooted simple triangulations (or equivalently, triangular
2-dissections with neither loop nor 2-cycle except the outer
face) with respectively a marked inner vertex, a marked inner edge,
and a marked inner edge incident to the root vertex. Let us also
denote \cvv, \cff, \cgg and \chh the corresponding generating series
according to half the number of triangular faces.

\medskip

\begin{figure}
  \centering
  \subfigure[]{\includegraphics[scale=.9,page=1]{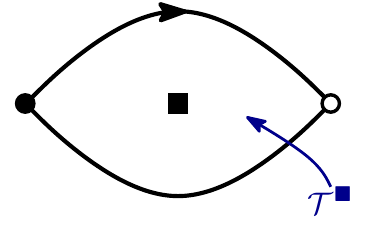}}\qquad
  \subfigure[]{\includegraphics[scale=.9,page=2]{Trig_U}}\qquad
  \subfigure[]{\includegraphics[scale=.9,page=3]{Trig_U}}\qquad
  \subfigure[]{\includegraphics[scale=.9,page=4]{Trig_U}}
  \caption{The four cases in the decomposition of an element of \cU.}
  \label{fig:Trig_U}
\end{figure}
Let us now proceed to the decomposition of an element of $\cU$, four
different configurations can appear (see Figure~\ref{fig:Trig_U}):
\begin{itemize}
\item[(a)] either the outer face is the only 2-cycle; 
\item[(b)] or the outer face is incident to a single face, the third
  side of which is a loop;
\item[(c)] or the outermost 2-cycle (other than the outer face) is
  incident to the root vertex;
\item[(d)] or the outermost 2-cycle (other than the outer face) is not
  incident to the root vertex.
\end{itemize}
Note that each inner edge of a simple triangulation, and in particular
of any map \sT in $\cTct$ is naturally oriented according to the
minimal 3-orientation of~\sT.
Hence there is a canonical way of plugging a rooted 2-dissection in
the (rooted) 2-cycle obtained by \emph{opening} the marked edge of a
map in \cTt.  Hence:
\[ \cU = \cTc + \cX \cdot (1+\cU) + \cTct \cdot \cU + (\cTt \setminus
\cTct)\cdot \cV,
\]
where \cX is an atom representing a triangle.  A similar decomposition
can be written for \cV, and putting things together and translating
them into the language of generating series, we obtain:
\[
\begin{cases}
\cuu = \cff + x(1+\cuu) + \chh\cuu +(\cgg - \chh)\cvv,&\\
\cvv = \cff + 2x(1+\cuu) + \cgg\cvv.&
\end{cases}
\]
The series $\cff$, $\cgg$ and $\chh$can now be expressed in
terms of $t$ as follows. A triangulation with $2n$ faces has $3n$
edges and $n+2$ vertices, so removing the non-root outer edge and
taking into account the marked inner vertex or edge, we get:
\[
\begin{cases}
\cff(x)=\sum_{n\geq 1}nt_nx^n=xt'(x), &\\
\cgg(x)=\sum_{n\geq 1}(3n-1)t_nx^n=3xt'(x)-t(x).&
\end{cases}
\]
The computation of $\chh(x)$ is a little more complicated. Let \sT be
a simple rooted triangulation with at least four faces. By merging the
two ends of the root edge of \sT and collapsing the inner triangle
incident to it into an edge, we obtain a 2-dissection that is
canonically rooted with same root vertex as \sT, and with a marked
edge incident to that vertex. This triangulation is not far from being
simple: it has no loop and the only 2-cycles separate the two marked
edges. Hence, decomposing this map along the sequence of nested
cycles, we get a sequence of elements of~\cTct, hence:
\[
\frac{t(x)-x}{x}= \sum_{n\geq 2}t_nx^n = \frac{\chh(x)}{1-\chh(x)},\quad
\text{from which we obtain}\quad \chh(x)=\frac{t(x)-x}{t(x)}.
\]
Combining all these equations and with the help of computer algebra,
we obtain the following expression for the generating series of rooted
quasi-simple triangular 1-dissections:
\begin{equation}\label{eq:trigright}
  d_3(x) = \frac{x \cdot [1+t(x)-2xt'(x)]}{1-2x+2t(x)-3xt'(x)+t(x)^2-3xt'(x)t(x)}.
\end{equation}
The following result is then obtained by gathering
Corollary~\ref{cor:trig} and Equations~\eqref{eq:trigleft}
and~\eqref{eq:trigright}, after some easy simplifications:
\begin{proposition}
  The generating series $t(x)$ of rooted simple triangulations
  according to half the number of faces is the unique solution of the
  following equation:
  \begin{equation}\label{eq:t}
    3x t'(x)^2 + 1 = [t(x) + 1] \cdot t'(x),
  \end{equation}
  which takes value 0 at 0. 
\end{proposition}
From this equation, written as $t'=(3xt'^2+1)/(t+1)$, one readily
extracts the development of $t(x)$ incrementally:
\[
t(x)=x+x^2+3x^3+13x^4+68x^5+399x^6+2530x^7+16965x^8+\dots
\]

As in the quadrangular case, the exact expression of the coefficients
(first obtained by Tutte from the recursive method~\cite{Tu63} and
subsequently recovered by Poulalhon and Schaeffer~\cite{S-these} using
a bijection with some decorated trees that are in bijection with
quaternary trees) can also be recovered from~\eqref{eq:t}:

\begin{corollary}
  For $n\geq 1$, the number of rooted simple triangulations with $2n$
  faces is equal to:
  \[
  \frac{2(4n-3)!}{n!(3n-1)!}.
  \]
  Equivalently, the series $t(x)$ is expressed as $t(x)=
  \dfrac{(\alpha(x)-2)(1-\alpha(x))}{\alpha(x)^2}$, where
  $\alpha\equiv\alpha(x)$ is the series of rooted quaternary trees,
  specified by $\alpha=1+x\alpha^4$.
\end{corollary}
\begin{proof}
  Equation~\eqref{eq:t} above admits a unique power series solution
  that is equal to $0$ at $0$, so it suffices to check that $f\equiv
  f(x):=(\alpha(x)-2)(1-\alpha(x)) / \alpha(x)^2$ is solution of~\eqref{eq:t}.
  Note that $x$ and $f(x)$ have rational expressions in terms of
  $\alpha$, and so does $f'(x)=\frac{\mathrm{d}f}{\mathrm{d}\alpha} /
  \frac{\mathrm{d}x}{\mathrm{d}\alpha}$:
  \[
  x=\frac{\alpha-1}{\alpha^4},\qquad f(x) =
  \dfrac{(\alpha(x)-2)(1-\alpha(x))}{\alpha(x)^2},\quad
  \text{and}\quad f'(x)=\alpha^2,
  \]
  and these expressions satisfy Equation~\eqref{eq:t}, which concludes
  the proof.

  Now as for Equation~\eqref{eq:q}, a direct proof without guessing
  the solution is possible; first let $r(x) = t'(x)$,
  Equation~\eqref{eq:t} rewrites $1+t = 3xr + 1/r$, hence $r = 3r'+3r
  -r'/r^2$, \ie $u'(3xu^2-1) + 2u^3 = 0$. Then we seek for $A$, $B$
  such that $\frac{\mathrm d}{\mathrm d x} \left(A(r) x + B(r) \right)
  = 0$. Easy computations show that $x r^{3/2} + r^{-1/2} $ is
  suitable. Using initial conditions, we get $x r^2+1 = r^{1/2}$, \ie,
  with $\alpha = x r^2 + 1$, $1+x\alpha^4 = \alpha$.
\end{proof}

\section{Radial distance of symmetric quadrangular dissections}\label{sec:quad_dista}
\label{sec:dist}
For $k\geq 2$, $i>0$, and \cDk a family of $k$-symmetric dissections,
let \cDki be the family of dissections in \cDk where the central
vertex is at distance $i$ from the outer face boundary; define the
size of a $k$-symmetric quadrangular (\resp triangular) dissection \sD
as the integer $n$ such that \sD has $kn$ inner faces (\resp $(2n+1)k$
inner faces).  Let $\Dki(x)$ be the generating series of \cDki with
respect to the size.

We compute here the expression of $\Dki(x)$ for general, simple, and
irreducible $k$-symmetric quadrangular $2k$-dissections and triangular
$k$-dissections.  So from now on \cDk is either a family of
$k$-symmetric quadrangular $2k$-dissections or a family of
$k$-symmetric triangular $k$-dissections.  Quadrangular and triangular
dissections are treated respectively in this section and in the next
one (Section~\ref{sec:dist_trig}).  Note that, when $k=2$, $\Dki(x)=0$
for quadrangular irreducible dissections and for triangular simple and
irreducible dissections, and when $k=3$, $\Dki(x)=0$ for triangular
irreducible dissections.

We use the letter $\cF$, $\cG$, $\cH$ (instead of $\cD$) for the
general, simple, and irreducible case respectively.  To obtain the
generating function expressions, we combine results of Bouttier
\etal~\cite{BoDFGu03,BoGu} on the $2$-point functions of general
quadrangulations 
with quotient and substitution operations, and
Lemma~\ref{lem:quotient_distance}.

\begin{rem}\label{rk:encloses}
  Lemma~\ref{lem:quotient_distance} implies that a $k$-symmetric
  quadrangular dissection has no cycle of length less than $2k$ that
  strictly encloses the central vertex (indeed the $k$-quotient has
  only faces of even degree, hence is bipartite, hence has no loop).
  And a $k$-symmetric triangular dissection has no cycle of length
  less than $k$ that strictly encloses the central vertex.
\end{rem}

\subsection{Symmetric quadrangular dissections}
 Define the algebraic generating function $P\equiv P(x)$ by
\begin{equation}\label{eq:defP}
  P=1+3xP^2,
\end{equation}
and define the algebraic generating function $X\equiv X(x)$ by
\begin{equation}\label{eq:defX}
  X+\frac{1}{X}+1=\frac{3}{P-1}.
\end{equation}
Define also $X_{\infty}:=P$. The following result is very closely
related to a result by Bouttier, Di Francesco, and Guitter~\cite{BoDFGu03}.

\begin{proposition}
  For each $i\geq 1$ and $k\geq 2$, the generating function $\Fki(x)$
  has the following expression (which does not depend on $k$):
\[
\Fki(x)=X_{i+1}-X_i,\ \ \ \mathrm{where}\ X_i=X_{\infty}\frac{(1-X^i)(1-X^{i+3})}{(1-X^{i+1})(1-X^{i+2})}.
\]
\end{proposition}
\begin{proof}
  If we take the $k$-quotient of dissection in $\cF_i^{(k)}$ we obtain
  a quadrangular 2-dissection with a pointed inner vertex
  at distance $i$ from the outer $2$-gon, according to
  Lemma~\ref{lem:quotient_distance}.  Notice that the outer $2$-cycle
  can be contracted (on the sphere) into a single edge $e$. This yields a
  quadrangulation of the sphere with a marked edge $e$ and pointed
  vertex $v$ at distance $i$ from $e$ (i.e., the extremity of $e$
  closest from $v$ is at distance $i$ from $v$). Hence $\Fki(x)$ is equal to the
  generating function $F_i(x)$ of quadrangulations with a marked
  edge $e$ and a marked vertex $v$ at distance $i$ from $e$.  The algebraic
  expression of $\Fki(x)=F_i(x)$, as $X_{i+1}-X_i$, has been obtained
  by Bouttier \etal~\cite{BoDFGu03}.
\end{proof}

\subsection{Symmetric simple quadrangular dissections}

\begin{figure}
  \centering
  \subfigure[The maximal 2-cycles,]{\includegraphics[page=1]{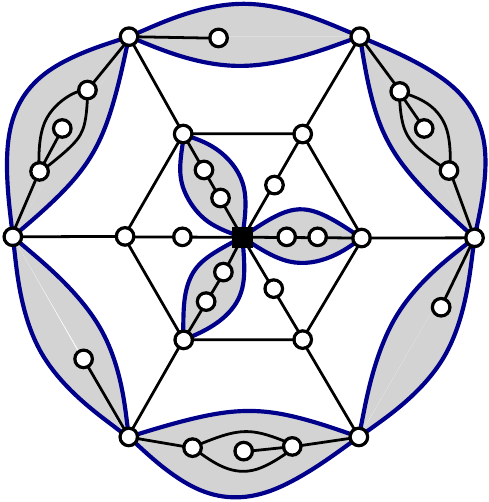}}
  \qquad\qquad
  \subfigure[and the simple core.]{\includegraphics[page=2]{simpleCore}}
  \caption{A 3-symmetric quadrangulation and its simple 3-symmetric core, obtained
    by collapsing all maximal 2-cycles into edges.}
  \label{fig:simplecore}
\end{figure}

We now compute $y\mapsto\Gki(y)$.  We classically proceed by
substitution (a substitution approach 
is also discussed in~\cite{BoGu} for rooted quadrangulations).  
The \emph{core} of $\gamma\in\cFki$ is
obtained by collapsing each maximal $2$-cycle (for the enclosed area)
of $\gamma$ into a single edge (note that no $2$-cycle can strictly
enclose the central vertex, according to Remark~\ref{rk:encloses}),
see Figure~\ref{fig:simplecore}.  This yields a simple dissection in
$\cGki$, and the distance of the central vertex to the outer boundary
is still $i$ (because there is no way of shortening this distance by
travelling inside a $2$-cycle). Conversely each $\gamma\in\cFki$ is
uniquely obtained from $\kappa\in\cGki$ ---with $nk$ inner faces---
where each of the $(2n+1)k$ edges is either left unchanged or blown
into a double edge in the interior of which a rooted quadrangulation
is patched, in a way that respects the symmetry of order $k$ (that is,
the $k$ edges in an orbit of edges of $\kappa$ undergo the same
substitution operation). Denoting by $f\equiv f(x)$ the series of
rooted quadrangulations according to the number of faces, we obtain
$\Fki(x)=\sum_{n\geq 1}[y^n]\Gki(y)\cdot x^n\cdot(1+f)^{2n+1}$. In other
words,
\begin{equation}\label{eq:FGquad}
  \Fki(x)=(1+f)\cdot\Gki(x\cdot(1+f)^2). 
\end{equation}
The series $f=f(x)$ is well known to be algebraic, 
having a rational expression in terms of $P$: $f=P(4-P)/3-1$ (we will
need this expression a few times).

Define $Q\equiv Q(y)$ as the algebraic series in $y$ defined by
\begin{equation}
  Q=1+yQ^3.
\end{equation}

\begin{lemma}\label{lem:xy_quad}
  Let $K(x)$ and let $L(y)$ be related by $K(x)=L(x(1+f)^2)$. Let
  $\widehat{K}(P)$ and $\widehat{L}(Q)$ be the expressions of $K(x)$
  and $L(y)$ in terms of $P\equiv P(x)$ and $Q\equiv Q(y)$, i.e.,
  $K(x)=\widehat{K}(P(x))$ and $L(y)=\widehat{L}(Q(y))$. Then
\[
\widehat{L}(Q)=\widehat{K}\Big(4-3/Q\Big).
\]
\end{lemma}
\begin{proof}
  The change of variable relation is $y=x(1+f)^2$.  We have
  \[
  y=x(1+f)^2=\frac{P-1}{3P^2}\Big(P\frac{4-P}{3}\Big)^2=\frac{P-1}{3}\Big(\frac{4-P}{3}\Big)^2
  \]
  Hence, if we write $Q\equiv Q(y)=3/(4-P(x))$, we have
  $y=(1-1/Q)Q^{-2}$, so that $Q=1+yQ^3$. In addition $P(x)=4-3/Q(y)$.
\end{proof}

For the rest of this subsection, when we have $y$ and $x$ together in
an equation we assume $y$ and $x$ to be related by $y=x(1+f)^2$.
Define $Y_{\infty}(y):=X_{\infty}(x)/(1+f)$, $Y_i(y):=X_i(x)/(1+f)$,
and $Y(y):=X(x)$.  The expression of $X_i(x)$ in terms of
$X_{\infty}(x)$ and $X(x)$ ensures that (we do not need
Lemma~\ref{lem:xy_quad} at this step):
\[
Y_i=Y_{\infty}\frac{(1-Y^i)(1-Y^{i+3})}{(1-Y^{i+1})(1-Y^{i+2})}.
\]
In addition, since $\Gki(y)=\Fki(x)/(1+f)$, we have
\[
\Gki(y)=Y_{i+1}-Y_i.
\]

We now apply Lemma~\ref{lem:xy_quad} to get an algebraic expression of
$\Gki(y)$, written in terms of $Q(y)$. We have
\[
Y_{\infty}(y)=\frac{X_{\infty}(x)}{1+f(x)}=\frac{3}{4-P}.
\]
Then Lemma~\ref{lem:xy_quad} ensures that $Y_{\infty}(y)=Q(y)$.
Lemma~\ref{lem:xy_quad} and the relation $X+1/X+1=3/(P-1)$ ensure
that $Y\equiv Y(y)$ is the algebraic generating function specified by
\begin{equation}\label{eq:Yquad}
  Y+\frac{1}{Y}=\frac{1}{Q-1}.
\end{equation}

To summarize we obtain:

\begin{proposition}
  For each $i\geq 1$ and $k\geq 2$, the generating function $\Gki(y)$
  has the expression (with $Y_{\infty}=Q$):
  \[
  \Gki(y)=Y_{i+1}-Y_i,\ \ \ \mathrm{where}\ Y_i=Y_{\infty}\frac{(1-Y^i)(1-Y^{i+3})}{(1-Y^{i+1})(1-Y^{i+2})}.
  \]
\end{proposition}

\subsection{Symmetric irreducible quadrangular dissections}

\begin{figure}
  \centering
  \subfigure[The maximal non empty 4-cycles,]{\qquad\includegraphics[page=1]{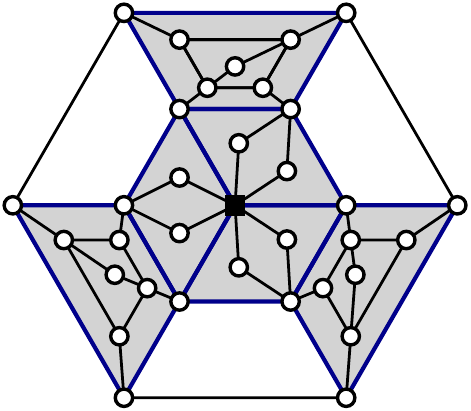}\qquad}
  \qquad
  \subfigure[and the irreductible core.]{\includegraphics[page=2]{irreductibleCore}}\qquad
  \caption{A 3-symmetric simple quadrangulation and its irreductible 3-symmetric core, obtained
    by emptying all maximal 4-cycles into faces.}
  \label{fig:irrcore}
\end{figure}

We now use a substitution approach at faces (instead of edges) to get
an expression for $\Hki(z)$ for $k\geq 3$ and $i>0$.
The \emph{core} of $\gamma\in\cGki$ is obtained by emptying each
maximal (for the enclosed area) $4$-cycle of $\gamma$ (note that no
$4$-cycle strictly encloses the central vertex, according to
Remark~\ref{rk:encloses}), see Figure~\ref{fig:irrcore}.  This yields a symmetric irreducible
dissection $\kappa\in\cHki$, and the distance of the pointed vertex to
the outer boundary is still $i$ (because there is no way of shortening
this distance by travelling inside a $4$-cycle). Conversely each
$\gamma\in\cGki$ is uniquely obtained from $\kappa\in\cHki$ 
 where at each face a rooted simple quadrangulation
(with at least one inner face) is patched, in a way that respects the
symmetry of order $k$ (i.e., the $k$ faces of an orbit undergo the
same patching operation).  Denoting by $g\equiv g(y)$ the series of
rooted simple non-degenerated quadrangulations according to the number of inner faces, we
obtain, for $k\geq 3$ and $i>0$:
\[
\Gki(y)=\Hki(g(y)). 
\]
Again the series $g\equiv g(y)$ is well known to be algebraic, having a rational
expression in terms of $Q$: 
$g=3Q-Q^2-2$.  

Define the algebraic series $R\equiv R(z)$ by
\begin{equation}
  R=z+R^2.
\end{equation}

\begin{lemma}\label{lem:yz_quad}
  Let $L(y)$ and let $M(z)$ be related by $L(y)=M(g(y))$. Let
  $\widehat{L}(Q)$ and $\widehat{M}(R)$ be the expressions of $L(y)$
  and $M(z)$ in terms of $Q\equiv Q(y)$ and $R\equiv R(z)$, i.e.,
  $L(y)=\widehat{L}(Q(y))$ and $M(z)=\widehat{M}(R(z))$. Then
  \[
  \widehat{M}(R)=\widehat{L}(R+1).
  \]
\end{lemma}
\begin{proof}
  The change of variable relation is $z=g(y)$.  We have
  \[
  z=g(y)=3Q-Q^2-2=-(Q-1)^2+(Q-1)
  \]
  Hence, if we write $R\equiv R(z)=Q(y)-1$, we have $z=-R^2+R$, so
  that $R=z+R^2$. In addition $Q(y)=R(z)+1$.
\end{proof}

For the rest of this subsection, when we have $z$ and $y$ together in
an equation we assume $z$ and $y$ to be related by $z=g(y)$.  Define
$Z_{\infty}(z):=Y_{\infty}(y)$, $Z_i(z):=Y_i(y)$, and $Z(z):=Y(y)$.
The expression of $Y_i(y)$ in terms of $Y_{\infty}(y)$ and $Y(y)$
ensures that (we do not need Lemma~\ref{lem:yz_quad} at this step):
\[
Z_i=Z_{\infty}\frac{(1-Z^i)(1-Z^{i+3})}{(1-Z^{i+1})(1-Z^{i+2})}.
\]
In addition, since $\Hki(z)=\Gki(y)$, we have
\[
\Hki(z)=Z_{i+1}-Z_i.
\]

We now apply Lemma~\ref{lem:yz_quad} to get an algebraic expression of
$\Hki(z)$, written in terms of $R(z)$. We have
\[
Z_{\infty}(z)=Y_{\infty}(y)=Q(y)=R(z)+1.
\]
Lemma~\ref{lem:yz_quad} and the relation $Y+1/Y=1/(Q-1)$ also ensure
that $Z\equiv Z(z)$ is the algebraic generating function specified by
\begin{equation}\label{eq:Zquad}
  Z+\frac{1}{Z}=\frac{1}{R}.
\end{equation}
  
To summarize we obtain:

\begin{proposition}
  For each $i\geq 1$ and $k\geq 3$, the generating function $\Hki(z)$
  has the expression (with $Z_{\infty}=R+1$):
  \[
  \Hki(z)=Z_{i+1}-Z_i,\ \ \ \mathrm{where}\ Z_i=Z_{\infty}\frac{(1-Z^i)(1-Z^{i+3})}{(1-Z^{i+1})(1-Z^{i+2})}.
  \]
\end{proposition}

\section{Radial distance of symmetric triangular dissections}\label{sec:dist_trig}

\subsection{Symmetric triangular dissections}
Define the algebraic generating function $P\equiv P(x)$ by
\begin{equation}\label{eq:defP_triang}
  P^2=1+8xP^3,
\end{equation}
and define the algebraic generating function $X\equiv X(x)$ by
\begin{equation}\label{eq:defX_triang}
  X+\frac{1}{X}+2=\frac{8}{P^2-1}.
\end{equation}
Define also $X_{\infty}\equiv X_{\infty}(x)$ and $A_{\infty}\equiv A_{\infty}(x)$ by 
\begin{equation}
  X_{\infty}=P,\ \ A_{\infty}\ \!\!^2=\frac{2P(P-1)}{(1+P)}.
\end{equation}
As for quadrangulations the following result is very closely
related to a result by Bouttier and Guitter~\cite{BoGu12}.

\begin{proposition}
  For each $i\geq 1$ and $k\geq 2$, the generating function $\Fki(x)$ has the following 
  expression (which does not depend on $k$):
\[
\Fki(x)=X_{i+1}-X_{i-1}+A_i\ \!\!^2-A_{i-1}\ \!\!^2,
\]
where
\[
X_i = X_{\infty} \cdot \frac{(1-X^i)(1-X^{i+2})}{(1-X^{i+1})^2}, \quad
A_i=A_{\infty}\cdot\Big(1-\frac{P+1}{4}X^i\frac{(1-X)(1-X^2)}{(1-X^{i+1})(1-X^{i+2})}\Big).
\]
\end{proposition}

\begin{proof}
  We recall the result in~\cite{BoGu12} about triangulations:
  \begin{itemize}
  \item the series $U_i$ of triangulations of the sphere with a
    pointed vertex $u$ and a marked edge $e=(v,v')$ with $v$ at
    distance $i$ and $v'$ at distance $i-1$ from $u$ is given by
    $U_i=X_i-X_{i-1}$,
  \item the series $V_i$ of triangulations of the sphere with a
    pointed vertex $u$ and a marked oriented edge $e=(v,v')$ with $v$
    and $v'$ at distance $i$ from $u$ is given by $V_i=A_i\
    \!\!^2-A_{i-1}\ \!\!^2$. (Note that the series $V_i$ decomposes as
    $V_{i,\mathrm{loop}}+V_{i,\mathrm{distinct}}$, whether the
    extremities of the marked edge are equal or distinct.)
  \end{itemize}

  The $k$-quotient of a dissection in $\cF_i^{(k)}$ is a triangular
  1-dissection $D$ with a pointed inner vertex $u$ at
  distance $i$ from the outer loop, according to
  Lemma~\ref{lem:quotient_distance}. The vertex incident to the outer
  loop is called the \emph{root-vertex} and denoted by $v$.  If we
  delete the outer loop we obtain a map $\widetilde{D}$ with an outer
  face of degree $2$ and all inner faces of degree $3$.  Two cases can
  occur: either the outer face contour of $\widetilde{D}$ is a $2$-gon
  or is made of two adjacent loops. In the first case let $v'$ be the
  other vertex of the $2$-gon: $v'$ is at distance either $i-1$,
  $i+1$, or $i$ from $u$, giving respective contributions
  $U_i$, $U_{i+1}$, $V_{i,\mathrm{distinct}}$. In the second case we have an
  ordered pair $D_1,D_2$ of 1-dissections, and one of
  these two dissections contains a marked vertex at distance $i$ from
  $v$. Orient the outer loop of $D_1$ clockwise and the outer loop of
  $D_2$ counterclockwise, and paste $D_1$ and $D_2$ together at
  their outer loops. This yields a triangulation of the sphere with a
  pointed vertex $u$ and a marked oriented loop whose incident vertex
  is at distance $i$ from $u$.  The corresponding contribution is
  $V_{i,\mathrm{loop}}$.  Gathering all cases we obtain
  $\Fki=U_i+U_{i+1}+V_{i,\mathrm{distinct}}+V_{i,\mathrm{loop}}=U_i+U_{i+1}+V_i$.
\end{proof}

\subsection{Symmetric simple triangular dissections}
 Call a rooted triangulation \emph{simply-rooted} if the root-edge is
not a loop.  To compute $y\mapsto\Gki(y)$, we proceed very similarly
as for quadrangulations.  Each $\gamma\in\cFki$ (for $k\geq 3$) is
uniquely obtained from $\kappa\in\cGki$ ---with $(2n+1)k$ inner
faces--- where each of the $(3n+2)k$ edges is either left unchanged or
blown into a double edge in the interior of which a simply-rooted
triangulation is patched, in a way that respects the symmetry of order
$k$ (that is, the $k$ edges in an orbit of edges of $\kappa$ undergo
the same substitution operation). Denoting by $f\equiv f(x)$ the
series of simply-rooted triangulations according to half the number of
faces, we obtain $\Fki(x)=\sum_{n\geq 1}[y^n]\Gki(y)\cdot
x^n\cdot(1+f)^{3n+2}$. In other words,
\begin{equation}\label{eq:FGtriang}
  \Fki(x)=(1+f)^2\cdot\Gki(x\cdot(1+f)^3). 
\end{equation}
The series $f=f(x)$ is algebraic, 
with a rational expression in terms of $P$: $1+f=-P(P^2-9)/8$.

Define $Q\equiv Q(y)$ as the algebraic series in $y$ given by
\begin{equation}
  Q=\frac{y}{(1-Q)^3},
\end{equation}
and define also $\tQ\equiv \tQ(y)$ as $\tQ:=(1+8Q)^{1/2}$.

\begin{lemma}\label{lem:xy_triang}
  Let $K(x)$ and $L(y)$ be related by $K(x)=L(x(1+f)^3)$. Let
  $\widehat{K}(P)$ and $\widetilde{L}(\tQ)$ be the expressions of
  $K(x)$ and $L(y)$ in terms of $P\equiv P(x)$ and $\tQ\equiv \tQ(y)$,
  i.e., $K(x)=\widehat{K}(P(x))$ and
  $L(y)=\widetilde{L}(\tQ(y))$. Then the expressions are the same,
  i.e.,
  \[
  \widetilde{L}=\widehat{K}.
  \]
\end{lemma}

\begin{proof}
  The change of variable relation is $y=x(1+f)^3$.  We have:
  \[
  y=x(1+f)^3=-\frac{1}{2^{12}}(P^2-1)(P^2-9)^3.
  \]
  Hence, if we write $Q\equiv Q(y)=(P(x)^2-1)/8$,    
  we have $y=Q(1-Q)^3$. In addition $P(x)=\tQ(y)$, hence $\widetilde{L}=\widehat{K}$.  
\end{proof}

Define (with $y$ and $x$ related by $y=x(1+f)^3$):
\[
Y_{\infty}(y):=\frac{X_{\infty}(x)}{(1+f)^2},\ Y_i(y):=\frac{X_i(x)}{(1+f)^2},\ Y(y):=X(x),
\]
and define
\[
B_{\infty}(y)=\frac{A_{\infty}(x)}{(1+f)},\ B_i(y):=\frac{A_i(x)}{(1+f)}.
\] 
The expression of $X_i(x)$ in terms of $X_{\infty}(x)$ and $X(x)$
ensures that
\[
Y_i=Y_{\infty}\frac{(1-Y^i)(1-Y^{i+2})}{(1-Y^{i+1})^2}.
\]
The expression of $A_i(x)$ in terms of $A_{\infty}(y)$ and $X(x)$ and
Lemma~\ref{lem:xy_triang} (to replace $P$ by $\tQ$ in the expression)
ensure that
\[
B_i=B_{\infty}\cdot\Big(1-\frac{\tQ+1}{4}Y^i\frac{(1-Y)(1-Y^2)}{(1-Y^{i+1})(1-Y^{i+2})}\Big).
\]

In addition, since $\Gki(y)=\Fki(x)/(1+f)^2$, we have
\[
\Gki(y)=Y_{i+1}-Y_{i-1}+B_i\ \!\!^2-B_{i-1}\ \!\!^2.
\]

Using Lemma~\ref{lem:xy_triang} we obtain (after simplifications)
\begin{equation}
  Y_{\infty}=\frac{1}{\tQ(1-Q)^2},\ \ B_{\infty}\ \!\!^2=\frac{16\cdot Q}{\tQ(1+\tQ)^2(1-Q)^2}.
\end{equation}

Lemma~\ref{lem:xy_triang} and the relation $X+1/X+1=8/(P^2-1)$ also
ensure that $Y\equiv Y(y)$ is the algebraic generating function
specified by
\begin{equation}\label{eq:Ytriang}
  Y+\frac{1}{Y}+2=\frac{1}{Q}.
\end{equation}
  
To summarize we obtain:

\begin{proposition}
  For each $i\geq 1$ and $k\geq 3$, the generating function $\Gki(y)$
  has the following expression (which does not depend on $k$):
  \[
  \Gki(y)=Y_{i+1}-Y_{i-1}+B_i\ \!\!^2-B_{i-1}\ \!\!^2,
  \]
  where
  \[
  Y_i=Y_{\infty}\frac{(1-Y^i)(1-Y^{i+2})}{(1-Y^{i+1})^2},\ \
  B_i=B_{\infty}\cdot\Big(1-\frac{\tQ+1}{4}Y^i\frac{(1-Y)(1-Y^2)}{(1-Y^{i+1})(1-Y^{i+2})}\Big).
  \]
\end{proposition}

\subsection{Symmetric irreducible triangular dissections} 
 We now use a substitution approach at faces to get an expression for
$\Hki(z)$ for $k\geq 4$ and $i>0$.  Similarly as in the quadrangulated
case, each $\gamma\in\cGki$ (for $k\geq 4$) is uniquely obtained from
$\kappa\in\cHki$ ---with $(2n+1)k$ inner faces--- where at each face a
rooted simple triangulation (with at least one inner face) is patched,
in a way that respects the symmetry of order $k$ (i.e., the $k$ faces
of an orbit undergo the same patching operation).  Denoting by
$g\equiv g(y)$ the series of rooted simple triangulations according to
half the number of faces, we obtain, for $k\geq 4$ and $i>0$:
\[
\Gki(y)=\frac{g}{y}\Hki(g^2/y). 
\]
Again the series $g\equiv g(y)$ is well known to be algebraic, having a rational
expression in terms of $Q$: 
$g=Q-2Q^2$.  

Define the algebraic series $R\equiv R(z)$ by
\begin{equation}
  R=\frac{z}{(1-R)^2}.
\end{equation}
Define also $\tR\equiv\tR(z)$ as $\tR:=\sqrt{1+9R}/\sqrt{1+R}$.

\begin{lemma}\label{lem:yz_triang}
  Let $L(y)$ and let $M(z)$ be related by $L(y)=M(g(y)^2/y)$.  Let
  $\widetilde{L}(\tQ)$ and $\widetilde{M}(\tR)$ be the expressions of
  $L(y)$ and $M(z)$ in terms of $\tQ\equiv \tQ(y)$ and $\tR\equiv
  \tR(z)$, i.e., $L(y)=\widetilde{L}(\tQ(y))$ and
  $M(z)=\widetilde{M}(\tR(z))$. Then the expressions are the same,
  i.e.,
  \[
  \widetilde{M}=\widetilde{L}.
  \]
\end{lemma}

\begin{proof}
  The change of variable relation is $z=g(y)^2/y$.  We have
  \[
  z=g(y)^2/y=Q\frac{(1-2Q)^2}{(1-Q)^3}.
  \]
  Hence, if we write $R\equiv R(z)=Q(y)/(1-Q(y))$, we have
  $z=R(1-R)^2$, so that $R=z/(1-R)^2$. In addition we have
  $Q(y)=R(z)/(1+R(z))$, so that
  $\tQ(y)=(1+8Q(y))^{1/2}=\sqrt{1+9R(z)}/\sqrt{1+R(z)}=\tR(z)$.  Hence
  $\widetilde{M}=\widetilde{L}$.
\end{proof}

Define (with $z$ and $y$ related by $z=g(y)^2/y$):
\[
Z_{\infty}(z):=\frac{y}{g}Y_{\infty}(y),\ Z_i(z):=\frac{y}{g}Y_i(y),\
Z(z):=Y(y),
\]
and define
\[
C_{\infty}(z)^2:=\frac{y}{g}B_{\infty}(y)^2,\
C_i(z)^2:=\frac{y}{g}B_i(y)^2.
\] 
The expression of $Y_i(y)$ in terms of $Y_{\infty}(y)$ and $Y(y)$
ensures that
\[
Z_i=Z_{\infty}\frac{(1-Z^i)(1-Z^{i+2})}{(1-Z^{i+1})^2}.
\]
The expression of $B_i(y)$ in terms of $B_{\infty}(y)$ and $Y(y)$ and
Lemma~\ref{lem:yz_triang} (to replace $\tQ$ by $\tR$ in the
expression) ensure that
\[
C_i=C_{\infty}\cdot\Big(1-\frac{\tR+1}{4}Z^i\frac{(1-Z)(1-Z^2)}{(1-Z^{i+1})(1-Z^{i+2})}\Big).
\]
In addition, since $\Hki(z)=\frac{y}{g}\Gki(y)$, we have
\[
\Hki(z)=Z_{i+1}-Z_{i-1}+C_i\ \!\!^2-C_{i-1}\ \!\!^2.
\]
Using Lemma~\ref{lem:yz_triang} we obtain (after simplifications):
\begin{equation}
  Z_{\infty}=\frac{1}{\tR(1-R)},\ \ C_{\infty}\ \!\!^2=\frac{16\cdot R}{(\tR+1)^2\tR(1-R^2)}.
\end{equation}
Lemma~\ref{lem:yz_triang} and the relation $Y+1/Y+2=1/Q$ also ensure
that $Z\equiv Z(z)$ is the algebraic generating function specified by
\begin{equation}\label{eq:Ztriang}
  Z+\frac{1}{Z}+1=\frac{1}{R}.
\end{equation}
To summarize we obtain:

\begin{proposition}
  For each $i\geq 1$ and $k\geq 4$, the generating function $\Hki(z)$
  has the following expression (which does not depend on $k$):
  \[
  \Hki(z)=Z_{i+1}-Z_{i-1}+C_i\ \!\!^2-C_{i-1}\ \!\!^2,
  \]
  where
  \[
  Z_i=Z_{\infty}\frac{(1-Z^i)(1-Z^{i+2})}{(1-Z^{i+1})^2},\ \
  C_i=C_{\infty}\cdot\Big(1-\frac{\tR+1}{4}Z^i\frac{(1-Z)(1-Z^2)}{(1-Z^{i+1})(1-Z^{i+2})}\Big).
  \]
\end{proposition}

\bibliographystyle{plain}
\bibliography{SymQuad}

\end{document}